\documentclass[12pt,reqno]{article}

\usepackage[usenames]{color}
\usepackage{amssymb, amsmath, amsthm}
\usepackage{graphicx}

\usepackage[colorlinks=true,
linkcolor=webgreen,
filecolor=webbrown,
citecolor=webgreen]{hyperref}

\definecolor{webgreen}{rgb}{0,.5,0}
\definecolor{webbrown}{rgb}{.6,0,0}

\usepackage{color}
\usepackage{fullpage}
\usepackage{amsfonts}
\usepackage{latexsym}
\usepackage{epsf}

\newcommand{\seqnum}[1]{\href{http://oeis.org/#1}{\underline{#1}}}

\def\OOO{\mathcal{O}}
\def\PPP{\mathcal{P}}
\def\NN{\mathbb{N}}
\def\RR{\mathbb{R}}
\def\ZZ{\mathbb{Z}}

\begin{document}

\begin{center}
\epsfxsize=4in
\end{center}

\theoremstyle{plain}
\newtheorem{proposition}{Proposition}

\begin{center}
\vskip 1cm{\LARGE\bf 
Implicit Divided Differences,\\
Little Schr\"oder Numbers,\\
and Catalan Numbers}
\vskip 1cm
\large
Georg Muntingh \\
Centre of Mathematics for Applications/\\
Department of Mathematics\\
University of Oslo\\
P.O. Box 1053, Blindern\\
N-0316, Oslo\\
Norway\\

\href{mailto:georgmu@math.uio.no}{\tt georgmu@math.uio.no}\\
\end{center}

\vskip .2 in

\begin{abstract}
Under general conditions, the equation $g(x,y) = 0$ implicitly defines $y$ locally as a function of $x$. In this short note we study the combinatorial structure underlying a recently discovered formula for the divided differences of $y$ expressed in terms of bivariate divided differences of $g$, by analyzing the number of terms $a_n$ in this formula. The main result describes six equivalent characterizations of the sequence $\{a_n\}$.
\end{abstract}

\section{Introduction}
The act of taking differences and dividing is a recursive process with a wide range of applications in mathematics. In approximation theory, such divided differences can be viewed as a discrete analogue of derivatives; see the survey by de Boor \cite{Boor05}. Generalizations of univariate divided differences play a role in the theory of orthogonal polynomials, Macdonald polynomials in particular, and are used to define recurrences between Schubert polynomials in the algebraic geometry of flag varieties \cite{Lascoux03}.

Although a classical topic dating back to Newton, interest in divided differences seems to have been revived recently. One point of view defines multivariate divided differences as the leading term of an interpolating polynomial and attempts to carry over properties from the univariate case \cite{Rabut00, Sauer07}. In addition, several formulas from (differential) calculus have been generalized to a `divided difference calculus', including several new chain rules \cite{Floater10, FloaterLyche07, WangWang06} akin to Fa\`a di Bruno's formula, and expressions for divided differences of inverse functions \cite{FloaterLyche08} and implicit functions \cite{MuntinghFloater11, Muntingh12}.

It is the goal of this paper to study the combinatorial structure underlying a formula by Muntingh and Floater \cite{MuntinghFloater11}, restated below as \eqref{eq:main}, for the divided differences of a function $y$ implicitly defined by a relation $g$. More precisely, for some open intervals $U, V\subset\RR$, let $y:U\longrightarrow V$ be a function that is implicitly defined by a function $g: U\times V\longrightarrow \RR$ via
\begin{equation}\label{eq:ImplicitlyDefined}
g\big(x,y(x)\big) = 0,\qquad
\frac{\partial g}{\partial y}\big(x,y(x)\big)\neq 0
\qquad \forall\ x\in U.
\end{equation}
Using a chain rule for divided differences, this relation induces relations between the divided differences of $y$ and those of $g$ and in particular yields \eqref{eq:main}. In this paper we study the number of terms $a_n$ in this formula.

In the following section, we recall the definition of a divided difference, a recurrence relation for the divided differences of $y$, and the explicit formula \eqref{eq:main}. Next, in Section \ref{sec:NumberOfTerms}, we first derive a generating function from the recurrence relation and then prove six equivalent characterizations of the sequence $\{a_n\}$. We end with an analysis of the asymptotic behaviour of $\{a_n\}$.

\section{Divided Differences}\label{sec:DividedDifferences}
Let $[x_0,\ldots,x_n]f$ denote the \emph{divided difference [of order $n$]} of a function
$f:(a, b)\longrightarrow\RR$ at the distinct points $x_0,\ldots,x_n \in (a,b)$, which is recursively defined by
$[x_0]f := f(x_0)$ and
\[ [x_0,\ldots, x_n]f =
  \frac{[x_1,\ldots, x_n]f - [x_0,\ldots, x_{n-1}]f}{x_n - x_0}
  \qquad \textup{if}~n>0. \]
For given indices $i_0,i_1,\ldots,i_k$ satisfying $i_0 < i_1 < \cdots < i_k$, we shall shorten notation to
$[i_0i_1\cdots i_k]f := [x_{i_0},x_{i_1},\ldots x_{i_k}]f$.

The above definitions generalize to bivariate divided differences as follows.
This time, let $f: U\longrightarrow \RR$ be defined on some rectangle
\[ U = (a_1,b_1)\times (a_2,b_2)\subset \RR^2.\]
Suppose we are given integers $m, n \ge 0$,
distinct points $x_0, \ldots, x_m \in (a_1, b_1)$,
and distinct points $y_0, \ldots, y_m \in (a_2, b_2)$.
The Cartesian product 
\[ \{ x_0, \ldots, x_m \} \times \{ y_0, \ldots, y_n \} \]
defines a rectangular grid of points in $U$.
The \emph{[bivariate] divided difference} of $f$ at this grid,
denoted by
\begin{equation}\label{eq:bivdd}
 [ x_0, \ldots, x_m; y_0, \ldots, y_n ]f,
\end{equation}
can be defined recursively as follows.
If $m = n = 0$, the grid consists of only one point
$(x_0,y_0)$, and we define
$[x_0;y_0]f := f(x_0,y_0)$
as the value of $f$ at this point.
In case $m > 0$, we can define \eqref{eq:bivdd} as
\[ \frac{[ x_1, \ldots, x_m; y_0, \ldots, y_n ]f - [ x_0, \ldots, x_{m-1}; y_0, \ldots, y_n ]f}{x_m - x_0}, \]
or if $n > 0$, as
\[ \frac{[ x_0, \ldots, x_m; y_1, \ldots, y_n ]f - [ x_0, \ldots, x_m; y_0, \ldots, y_{n-1} ]f}{y_n - y_0}. \]
If both $m>0$ and $n>0$
the divided difference \eqref{eq:bivdd} is uniquely defined by either
recurrence relation. Similarly to the univariate case, we shorten the notation for bivariate divided differences to
\[ [ i_0i_1\cdots i_s; j_0 j_1 \cdots j_t]f
:= [x_{i_0},x_{i_1},\ldots,x_{i_s};y_{j_0},y_{j_1},\ldots,y_{j_t}]f. \]

In what follows, assume that $y$ and $g$ are related by \eqref{eq:ImplicitlyDefined}. Applying a chain rule for divided differences, one can derive \cite{MuntinghFloater11}
\begin{equation}\label{eq:RecursionFormulaZ}
[01]y = - \frac{[01;1]g}{[0;01]g},
\end{equation}
and for any $n\geq 2$ the recurrence relation
\begin{equation}\label{eq:RecursionFormulaA}
[01\cdots n]y =  - \sum_{k = 2}^n\ \sum_{0 = i_0 < \cdots < i_k = n}\ \sum_{\substack{s = 0\\s = i_s - i_0}}^k \frac{[01\cdots s;i_s i_{s+1}\cdots i_k]g}{[0;0n]g}\!\prod_{l = s + 1}^k [i_{l-1} (i_{l-1} + 1)\cdots i_l]y.
\end{equation}
The third summation might look a bit mysterious. It is a concise way of describing all $s$ for which the increasing sequence $0 = i_0 < i_1 < \cdots < i_k = n$ starts with at least $s$ steps of 1, i.e., $i_1 - i_0 = \cdots = i_s - i_{s-1} = 1$.

\begin{figure}[t]
\begin{center}
\includegraphics[scale=0.815]{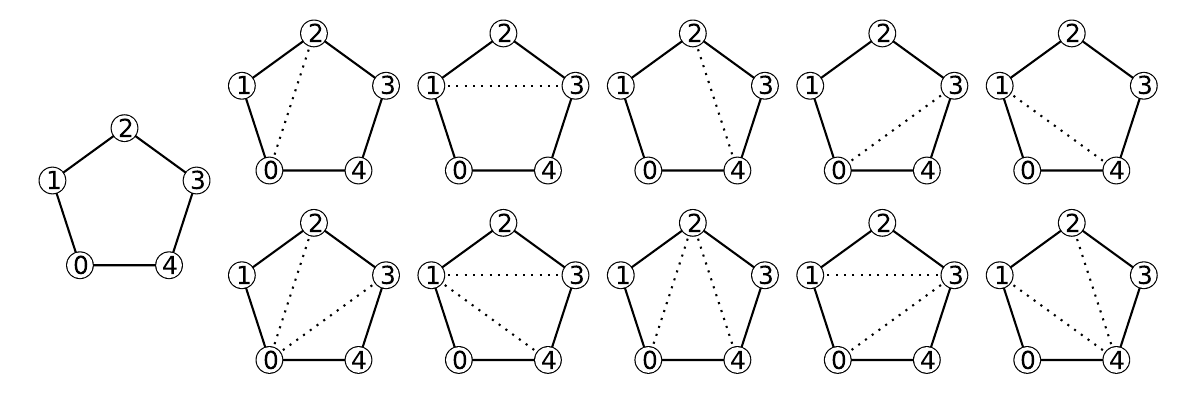}
\end{center}
\caption{The dissections of a convex polygon with vertices labeled $0, 1, 2, 3, 4$. The bottom-right dissection has faces $(0,1,4)$, $(1,2,4)$, and $(2,3,4)$.}\label{fig:PolygonPartitions}
\end{figure}

For $n\geq 2$, this recurrence relation was ``solved'' to yield the explicit formula \eqref{eq:main}. In order to describe this formula, we introduce some notation for dissections of a convex polygon. With a sequence of labels $0,1,\ldots,n$ we associate the ordered vertices of a convex polygon. A \emph{dissection of a convex polygon} is the result of connecting any pairs of nonadjacent vertices with straight line segments, none of which intersect. We denote the set of all such dissections of a polygon with vertices $0,1,\ldots,n$ by $\PPP(0,1,\ldots,n)$. The vertices $0, 1, \ldots, n$ and line segments (between either adjacent or nonadjacent vertices) form the vertices and edges of a plane graph. As such, every dissection $\pi\in \PPP(0, 1,\ldots, n)$ is described by its set $F(\pi)$ of \emph{[oriented] faces}, which does not include the unbounded face. Each face $f\in F(\pi)$ is represented by a subsequence $f = (i_0, i_1,\ldots, i_k)$ of the sequence $(0, 1, \ldots, n)$ of length at least three.

See Figure \ref{fig:PolygonPartitions} for the $11$ dissections of a convex polygon with vertices $0, 1, 2, 3, 4$. For general $n$, this number is given by the \emph{little Schr\"oder number} \cite{Stanley97} (or \emph{super-Catalan number})
\[\frac{1}{n-1} \sum_{k=1}^{n-1} {n-1\choose k}{n-1+k\choose k-1},\]
which form sequence \seqnum{A001003} in The On-Line Encyclopedia for Integer Sequences (OEIS)~\cite{Sloane12}.

The main result of \cite{MuntinghFloater11} is the formula
\begin{equation}\label{eq:main}
[01\cdots n]y = \sum_{\pi\in \PPP(0, \ldots,n)} \prod_{(i_0, \ldots,i_k)\in F(\pi)} \sum_{\substack{s = 0 \\s = i_s - i_0}}^k - \frac{[i_0 \cdots i_s;i_s \cdots i_k]g}{[i_0; i_0 i_k]g} \!\prod_{\substack{l = s + 1\\ i_l - i_{l-1} = 1}}^k \!\left( - \frac{ [i_{l-1} i_l; i_l]g}{[i_{l-1}; i_{l-1} i_l]g} \right),
\end{equation}
which expresses divided differences of $y$ solely in terms of bivariate divided differences of $g$. In a follow-up paper, this formula was generalized to a formula for divided differences of \emph{multivariate} implicit functions \cite{Muntingh11, Muntingh12}.

\section{The Number of Terms}\label{sec:NumberOfTerms}
For any implicit divided differences $[01\cdots n]y$ of any order $n\geq 2$, we wish to analyze the number of terms $a_n$ in the right hand side of Equation \eqref{eq:main}. Equation \eqref{eq:RecursionFormulaZ} yields the natural extension $a_1 = 1$. With the additional convention $a_0 = 1$, we can collect these integers into a sequence $\{a_n\}$. The sequences in this paper will always start at $n = 0$, unless stated otherwise. Alternatively, we can encode the integers $a_n$ by an \emph{[ordinary] generating function}, which is the formal power series $\sum_{n=0}^\infty a_n x^n$. Table \ref{tab:NumberOfTerms} lists $a_n$ for small values of $n$, and more terms can be found as sequence \seqnum{A162326} in The OEIS. 

\begin{table}[t]
\begin{center}
\begin{tabular}{cccccccccccccccccccccccccccc}
\hline\hline
  $n$ & : & 0 & 1 & 2 &  3 &  4 &   5 &    6 &     7 &      8 &       9 &      10 &       11\\
  \hline
$a_n$ & : & 1 & 1 & 3 & 13 & 71 & 441 & 2955 & 20805 & 151695 & 1135345 & 8671763 & 67320573\\
\hline\hline
\end{tabular}
\end{center}
\caption{The number of terms $a_n$ in the right hand side of Equation \eqref{eq:main} for small values of $n$, with the additional convention $a_0 = a_1 = 1$.}\label{tab:NumberOfTerms}
\end{table}

To analyze $\{a_n\}$, we rewrite \eqref{eq:RecursionFormulaA} in terms of the more standard notion of a composition. Let $\NN = \{1, 2, \ldots \}$ denote the set of positive integers. A \emph{composition [of $n$ of length $k$]} is any sequence $(j_1,\ldots,j_k)\in \NN^k$ satisfying $j_1 + \cdots + j_k = n$. The positive integers $j_1,\ldots,j_k$ are called the \emph{parts} of the composition. Clearly these correspond to increasing sequences $0 = i_0 < \cdots < i_k = n$ via $j_l = i_l - i_{l-1}$ for $l = 1, \ldots, k$. Thus \eqref{eq:RecursionFormulaA} is equivalent to
\begin{equation}\label{eq:RecursionFormulaB}
[01\cdots n]y =  - \sum_{k = 2}^n\ \sum_{\substack{\text{compositions}\\ j_1 + \cdots + j_k = n}}\ \sum_{\substack{s = 0\\s = j_1 + \cdots + j_s}}^k
\end{equation}
\[ \hfill \frac{\big[01\cdots s;(s + j_{s+1})\cdots n\big]g}{[0;0n]g}\!\prod_{l = s + 1}^k \big[(j_1 + \cdots + j_{l-1})\cdots (j_1 + \cdots + j_l)\big]y. \]

\begin{proposition}\label{prop:GeneratingFunction}
The sequence $\{a_n\}$ has associated generating function 
\begin{equation}\label{eq:GeneratingFunction}
G(x) := \sum_{n=0}^\infty a_n x^n = 
\frac54 - \frac14 \sqrt{\frac{1 - 9x}{1 - x}}.
\end{equation}
\end{proposition}

\begin{proof}
Replacing, in Equation \eqref{eq:RecursionFormulaB}, the divided differences of $y$ of order $n$ by the number of terms $a_n$ yields
\[ a_n = \sum_{k=2}^n\ \sum_{\substack{\text{compositions} \\j_1 + \cdots + j_k = n}}\ \sum_{\substack{s=0\\ s = j_1 + \cdots + j_s}}^k\ \prod_{l=1}^k a_{j_l}, \qquad n\geq 2. \]
The third sum makes sure that any composition starting with precisely $s$ ones will be counted $s+1$ times. We can, therefore, replace this sum by a sum over the number $t$ of initial parts in the composition that are one. Using that $a_1 = 1$, we find
\[ a_n = 1 + \sum_{k=2}^n \ \sum_{t=0}^{k-1} \ \sum_{\substack{\text{compositions}\\j_1 + \cdots + j_{k-t} = n-t}} \prod_{l = 1}^{k-t} a_{j_l}, \qquad n\geq 2,\]
where the first term, 1, is included to make sure that the composition $n = 1 + 1 + \cdots + 1$ is counted $n+1$ times. The generating function $G$ will therefore satisfy
\begin{align*}
G & = \frac{1}{1 - x} + \sum_{k = 2}^\infty \sum_{t = 0}^{k - 1} x^t (G - 1)^{k-t}\\
  & = \frac{1}{1 - x} + (G - 1)\sum_{k = 1}^\infty \sum_{t = 0}^k x^t (G - 1)^{k-t}\\
  & = \frac{1}{1 - x} + (G - 1)\left[\frac{1}{1-x} \frac{1}{2 - G} - 1\right].
\end{align*}
Multiplying by $(1 - x)(2 - G)$ and rewriting gives
\begin{equation}\label{eq:MinimumPolynomial}
2G^2 - 5G + 2 + \frac{1}{1 - x} = 0.
\end{equation}
Since $G(0) = 1$, we find that $G$ is as in the Proposition.
\end{proof}

Up to a constant and a rescaling, the generating function $G(x)$ is exactly the magnetization $M_\infty(s)$ for the XYZ spin chain along a special line of couplings \cite{FendleyHagendorf10}.

Before we state the next proposition, we need to recall some standard notions. The \emph{Gaussian hypergeometric series [with parameters $a,b,c$]} is defined by 
\[_2F_1(a,b;c;z) = \sum_{n=0}^\infty \frac{(a)_n(b)_n}{(c)_n} \, \frac{z^n}{n!},\]
where the \emph{Pochhammer symbol} $(\cdot)_n$ is defined by
\[ (a)_n =
\left\{
 \begin{array}{ll}
  1,                     & \text{if } n = 0; \\
  a(a+1) \cdots (a+n-1), & \text{if } n > 0.
 \end{array}
\right.
\]
Hypergeometric series with adjacent parameter values are related \cite[Section 15.2.11]{AbramowitzStegun64} by Gauss' contiguous relation
\begin{equation}\label{eq:HypergeometricIdentity}
(c-b) _2F_1(a,b-1;c;z) + (2b - c + (a - b)z) _2F_1(a,b;c;z) + b(z-1) _2F_1(a,b+1;c;z) = 0.
\end{equation}
Moreover, as explained by Roy \cite{Roy87}, sums of products of binomial coefficients can often be reduced to hypergeometric series, with the help of elementary identities like
\begin{equation}\label{eq:PochhammerIdentities}
(a)_{2k} = 2^{2k} \left(\frac{a}{2}\right)_k \left(\frac{a+1}{2}\right)_k,\qquad \frac{n!}{(n-k)!} = (-1)^k (-n)_k.
\end{equation}

Given a sequence $\{a_n\}$, define its \emph{binomial transform} $\{b_n\}$ by 
\[ b_n = \sum_{k=0}^n {n\choose k} a_k,\qquad n\geq 0. \]
It is well known that if $A(x)$ is the generating function of an integer sequence $\{a_n\}$, then \[ \frac{1}{1-x}A\left(\frac{x}{1-x}\right)\]
is the generating function of its binomial transform. Although sometimes proved in the case that $A(x)$ is analytic \cite[Appendix]{Boyadzhiev09}, the statement holds for any formal power series $A(x)$~\cite{Gould90}.

For any $n\geq 0$, let
\[ C_n := \frac{1}{n+1}{2n\choose n}\]
denote the \emph{$n$-th Catalan number}, which form sequence \seqnum{A000108} in The OEIS. It is well known \cite{LarcombeWilson01} that the sequence $\{C_n\}$ has generating function
\begin{equation}\label{eq:CatalanGeneratingFunction}
F(x) := \frac{1 - \sqrt{1 - 4x}}{2x}.
\end{equation}
By means of the binomial transform, we shall relate $\{a_n\}$ to $\{2^n C_n\}$, which form sequence \seqnum{A151374} in The OEIS.

\begin{proposition}\label{thm:MainProposition}
With the convention $a_0 = a_1 = 1$, the following statements are equivalent and hold.
\begin{enumerate}
\item[(a)] The number of terms in the r.h.s. of \eqref{eq:main} is $a_n$ for any $n\geq 2$;
\item[(b)] The sequence $\{a_n\}$ has generating function given by \eqref{eq:GeneratingFunction};
\item[(c)] The integers $a_n$ satisfy the quadratic recurrence relation
\[ a_n = 1 + 2\sum_{m=1}^{n-1} a_m a_{n-m},\qquad n \geq 2;\]
\item[(d)] The integers $a_n$ satisfy the linear recurrence relation
\[ na_n = (-14 + 10n)a_{n-1} + (18 - 9n)a_{n-2},\qquad n \geq 2;\]
\item[(e)] One has $\displaystyle a_n = \,_2F_1\left(\frac 1 2, 1-n; 2; -8\right)$ for any $n\geq 2$;
\item[(f)] The sequence $\{a_{n+1}\}_{n\geq 0}$ is the binomial transform of $\big\{2^n C_n\big\}_{n\geq 0}$.
\end{enumerate}
\end{proposition}
\begin{proof}
Clearly (a) holds by definition of the $\{a_n\}$, and (a) $\Longleftrightarrow$ (b) by Proposition \ref{prop:GeneratingFunction}.

``(b) $\Longleftrightarrow$ (c)'': Given that $a_0 = 1$ and $a_1 = 1$, the sequence $\{a_n\}$ has generating function $G$ given by \eqref{eq:GeneratingFunction} if and only if \eqref{eq:MinimumPolynomial} holds, i.e., if and only if
\begin{align*}
 0 & = 2 \left(\sum_{n = 0}^\infty a_n x^n \right)^2 - 5\sum_{n = 0}^\infty a_n x^n + 2 + \sum_{n=0}^\infty x^n\\
   & = \sum_{n = 0}^\infty \left(1 - 5a_n + 2 \sum_{m = 0}^n a_m a_{n-m}\right) x^n + 2\\
   & = \sum_{n = 2}^\infty \left( 1 - a_n + 2\sum_{m = 1}^{n-1} a_m a_{n-m} \right) x^n,
\end{align*}
which happens precisely when $\{a_n\}$ satisfies (c).

``(b) $\Longrightarrow$ (d)'': In the differential ring of formal power series $\ZZ[[x]]$ with derivation $\frac{\text{d}}{\text{d} x}$, the generating function $G$ satisfies the differential equation
\[	 0 = \big(1 - 10x + 9x^2\big)G' + 4G - 5,\]
which can be written in terms of the $\{a_n\}$ as
\begin{align*}
0 = &\ (-5 + 4a_0 + a_1)\cdot 1 + (2 a_2 - 6 a_1)\cdot x \\
    &\ + \sum_{n = 2}^\infty \Big((n+1) a_{n+1} + (4 - 10n) a_n + 9(n-1)a_{n-1} \Big) \cdot x^n.
\end{align*}
Comparing coefficients yields (d).

``(d) $\Longrightarrow$ (e)'': Substituting $a = \frac12, b = 2 - n, c = 2, z = -8$ in Equation \eqref{eq:HypergeometricIdentity}, one finds that $_2F_1\left(\frac 1 2, 1-n; 2; -8\right)$
satisfies the same recurrence relation as $a_n$, and must therefore be equal to $a_n$.

``(e) $\Longrightarrow$ (f)'': The binomial transform $\{b_n\}$ of $\{2^n C_n\}$ is given by
\[ b_n = \sum_{k = 0}^n {n\choose k} 2^k \frac{1}{k+1} {2k\choose k}, \qquad n\geq 0. \]
Using the identities \eqref{eq:PochhammerIdentities} with $a = 1$, one finds
\begin{align*}
b_n & = \sum_{k=0}^n \frac{(2k)!}{k!} \cdot \frac{n!}{(n-k)!} \cdot \frac{1}{(k+1)!} \cdot \frac{2^k}{k!} \\ 
    & = \sum_{k=0}^n \frac{2^{2k} (\frac{1}{2})_k (1)_k }{(1)_k} \cdot (-1)^k (-n)_k \cdot \frac{1}{(2)_k} \cdot \frac{2^k}{k!}\\
    & = \sum_{k=0}^n \frac{(\frac12 )_k (-n)_k}{(2)_k} \, \frac{(-8)^k}{k!} \\
    & = \,_2F_1\left(\frac12, - n; 2; -8\right)\\
    & = a_{n+1}
\end{align*}
for any $n\geq 0$.

``(f) $\Longrightarrow$ (b)'': Since $\{C_n\}$ has generating function $F(x)$ given by \eqref{eq:CatalanGeneratingFunction}, the sequence $\{2^n C_n\}$ has generating function
\[ H(x) := F(2x) = \frac{1 - \sqrt{1 - 8x}}{4x},\]
and its binomial transform $\{a_{n+1}\}_{n\geq 0}$ has generating function
\[ \sum_{n = 0}^\infty a_{n+1} x^n
 = \frac{1}{1-x}H\left(\frac{x}{1-x}\right)
 = \frac{1 - \sqrt{1 - \frac{8x}{1-x}}}{4x}
 = \frac{1 - \sqrt{\frac{1 - 9x}{1 - x}}}{4x}. \]
We conclude that $\{a_n\}$ has generating function
\[ \sum_{n=0}^\infty a_n x^n = 1 + x \left( \frac{1 - \sqrt{\frac{1-9x}{1-x}}}{4x} \right)
 = \frac54 - \frac14 \sqrt{\frac{1-9x}{1-x}} = G(x). \qedhere \]
\end{proof}

We end this paper with a remark on the asymptotic behaviour of the sequence $\{a_n\}$. Flajolet and Sedgewick \cite{FlajoletSedgewick09} describe how the asymptotic properties of a sequence can be deduced from the \emph{singularities} of its generating function, i.e., points where the function ceases to be analytic. The location $\rho$ of the singularity closest to the origin dictates an exponential growth $\rho^{-n}$, while the singularity type determines a subexponential growth factor $C(n)$. Using the dictionary in Figure VI.5 in \cite{FlajoletSedgewick09}, the expansion
\[
G(x) = \frac{5}{4} - \frac{3}{8\sqrt{2}}\sqrt{1 - 9x} + \OOO\left(\Big(x - \frac19 \Big)^{3/2}\right)  \, 
\]
around the singularity at $\rho = \frac19$ yields the asymptotic formula
\begin{equation}\label{eq:AsymptoticFormula}
C(n)\cdot 9^n,\qquad C(n) = \frac{3}{16\sqrt{2\pi}}  n^{-3/2} + \OOO(n^{-5/2}),
\end{equation}
for the coefficients $a_n$. Using Proposition \ref{thm:MainProposition}(d) to quickly calculate the first 1000 entries in $\{a_n\}$, we compare \eqref{eq:AsymptoticFormula} to $a_n$ by plotting the relative error
\[ 1 - \frac{C(n)\cdot 9^n}{a_n}\]
in Figure \ref{fig:Asymptotics}.

\begin{figure}
\begin{center}
\includegraphics[scale=0.93]{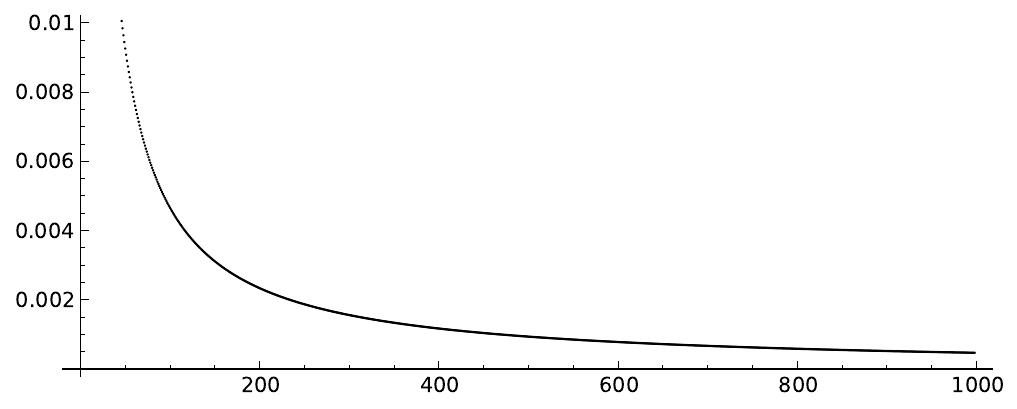}
\end{center}
\caption{The relative error $1 - C(n)\cdot 9^n/a_n$ plotted against $n$.}\label{fig:Asymptotics}
\end{figure}

\section{Acknowledgments}
I am grateful to Paul Kettler, for kindly providing several useful comments on a draft of this paper, and to Paul Fendley and Emeric Deutsch, for conjecturing the generating function $G(x)$ to me in a personal communication.

\bigskip
\hrule
\bigskip

\noindent 2010 {\it Mathematics Subject Classification}:
Primary 05A15; Secondary 26A24.

\bigskip

\noindent \emph{Keywords: } 
divided differences, implicit functions, polygon dissections, generating functions, hypergeometric series, binomial transform.

\bigskip
\hrule
\bigskip

\noindent (Concerned with sequences
\seqnum{A000108}, \seqnum{A001003}, \seqnum{A151374}, \seqnum{A162326}.)

\bigskip
\hrule
\bigskip





\begin{thebibliography}{10}
\bibitem{AbramowitzStegun64} M. Abramowitz, I. A. Stegun, \emph{Handbook of Mathematical Functions with Formulas, Graphs, and Mathematical Tables}, National Bureau of Standards Applied Mathematics Series {\bf 55}, For sale by the Superintendent of Documents, U.S. Government Printing Office, Washington, D.C., 1964.

\bibitem{Boor05} C. de Boor, Divided differences, \emph{Surv. Approx. Theory} {\bf 1} (2005), 46--69 (electronic).

\bibitem{Boyadzhiev09} K. N. Boyadzhiev, Harmonic number identities via Euler's transform, \emph{J. Integer Seq.} {\bf 12} (2009).

\bibitem{FendleyHagendorf10} P. Fendley, C. Hagendorf, Exact and simple results for the XYZ and strongly interacting fermion chains, \emph{J. Phys. A: Math. Theor.} {\bf 43} (2010).

\bibitem{FlajoletSedgewick09} P. Flajolet and R. Sedgewick, \emph{Analytic Combinatorics}, Cambridge University Press, 2009.

\bibitem{Floater10} M. S. Floater, A chain rule for multivariate divided differences, \emph{BIT Num. Math.} {\bf 50} (2010), 577--586.

\bibitem{FloaterLyche07} M. S. Floater and T. Lyche, Two chain rules for divided differences and Fa\`a di Bruno's formula, \emph{Math. Comp.} {\bf 76} (2007), 867--877 (electronic).

\bibitem{FloaterLyche08} M. S. Floater and T. Lyche, Divided differences of inverse functions and partitions of a convex polygon, \emph{Math. Comp.} {\bf 77} (2008), 2295--2308.

\bibitem{Gould90} H. W. Gould, Series transformations for finding recurrences for sequences, \emph{Fibonacci Quart.} {\bf 28} (1990), 166--171.

\bibitem{LarcombeWilson01} P. J. Larcombe, P. D. C. Wilson, On the generating function of the Catalan sequence: a historical perspective, \emph{Congr. Numer.} {\bf 149} (2001), 97--108.

\bibitem{Lascoux03} A. Lascoux, \emph{Symmetric Functions and Combinatorial Operators on Polynomials}, CBMS Regional Conference Series in Mathematics {\bf 99}, Published for the Conference Board of the Mathematical Sciences, 2003.

\bibitem{MuntinghFloater11} G. Muntingh and M. S. Floater, Divided differences of implicit functions, \emph{Math. Comp.} {\bf 88} (2011), 2185--2195.

\bibitem{Muntingh11} G. Muntingh, \emph{Topics in Polynomial Interpolation Theory}, Ph.D. Thesis at the Centre of Mathematics for Applications, University of Oslo, 2011.

\bibitem{Muntingh12} G. Muntingh, Divided differences of multivariate implicit functions, \emph{BIT Num. Math.}, {\bf 52} (2012), Pages 703--723.

\bibitem{Rabut00} C. Rabut, Multivariate divided differences with simple knots, \emph{SIAM J. Numer. Anal.} {\bf 38} (2001), 1294--1311.

\bibitem{Roy87} R. Roy, Binomial identities and hypergeometric series, \emph{Amer. Math. Monthly} {\bf 94} (1987), 36--46.

\bibitem{Sauer07} T. Sauer, Degree reducing polynomial interpolation, ideals and divided differences, in \emph{Curve and surface fitting: Avignon 2006}, Mod. Methods Math., Nashboro Press, 2007, pp.\ 220--237.

\bibitem{Sloane12} N. J. A. Sloane et al., \emph{The On-Line Encyclopedia of Integer Sequences}, Published electronically at \texttt{http://oeis.org} (2012).

\bibitem{Stanley97} R. P. Stanley, Hipparchus, Plutarch, Schr\"oder, and Hough, \emph{Amer. Math. Monthly} {\bf 104} (1997), 344--350.

\bibitem{Stanley99} R. P. Stanley, \emph{Enumerative Combinatorics. Vol. 2}, Cambridge Studies in Advanced Mathematics {\bf 62}, Cambridge University Press, 1999.

\bibitem{WangWang06} X.-H. Wang, H.-Y. Wang, On the divided difference form of Fa\`a di Bruno's formula, \emph{J. Comput. Math.} {\bf 24} (2006), 553--560.
\end{thebibliography}
\end{document}